\RequirePackage{fix-cm}
\documentclass[smallextended]{svjour3}       
\smartqed  

\journalname{}
\usepackage{booktabs}
\usepackage{amsmath}
\usepackage{amssymb,epsfig,multirow}
\usepackage{float}
\usepackage{color}
\usepackage{enumerate}
\usepackage{graphicx,subfigure,array}
\usepackage{algorithmicx,algorithm}
\usepackage[justification=centering]{caption}
\numberwithin{equation}{section}
\newtheorem{lem}{Lemma}[section]

\newtheorem{thm}{Theorem}[section]
\newtheorem{cor}{Corollary}[section]
\newtheorem{prop}{Proposition}[section]

\newtheorem{exam}{Example}[section]
\newtheorem{asmp}{Assumption}

\def\R{{\mathbb R}}

\begin{document}
\title{On local minimizers of generalized trust-region subproblem
\thanks{This research was supported by the Beijing Natural Science Foundation, grant Z180005 and by the National Natural Science Foundation of China under grants 12171021, and 11822103.}
}

\titlerunning{Local minimizers of generalized TRS}        

\author{Jiulin Wang
 \and Mengmeng Song
\and Yong Xia}


\institute{
J. Wang
 \and M. Song  \and Y. Xia \at
              LMIB of the Ministry of Education, School of Mathematical Sciences, Beihang University, Beijing, 100191, P. R. China \\
              \email{(wangjiulin@buaa.edu.cn (J. L. Wang); songmengmeng@buaa.edu.cn (M. M. Song); yxia@buaa.edu.cn (Y. Xia, corresponding author) )}
}

\date{Received: date / Accepted: date}

\maketitle
\begin{abstract}
Generalized trust-region subproblem (GT) is a nonconvex quadratic optimization with a single quadratic constraint. It reduces to the classical
trust-region subproblem (T) if the constraint set is a Euclidean ball. (GT) is polynomially solvable based on its inherent hidden convexity.  In this paper, we study local minimizers of (GT). Unlike (T) with at most one local nonglobal minimizer, we can prove that two-dimensional (GT) has at most {\it two} local nonglobal minimizers, which are shown by example to be attainable.
The main contribution of this paper is to prove that, at {\it any} local nonglobal minimizer of (GT),  not only
the strict complementarity condition holds, but also  the standard second-order sufficient optimality condition remains necessary.
As a corollary, finding {\it all} local nonglobal minimizers of (GT) or proving the nonexistence can be done in polynomial time. Finally, for (GT) in complex domain, we prove that there is  {\it no} local nonglobal minimizer, which  demonstrates that real-valued optimization problem may be more difficult to solve than its complex version.

\end{abstract}
\section{Introduction}\label{sec1}
Given two quadratic functions
\[
f(x)=x^TAx+2a^Tx,~g(x)=x^TBx+2b^Tx+c,
\]
where $A,B\in \mathbb{R}^{n \times n}$ are symmetric and possibly not positive semidefinite, $a,b\in \mathbb{R}^{n}$ and $c\in \mathbb{R}$,
we consider the inequality and equality constrained versions of generalized  trust-region subproblem:
\begin{eqnarray*}
{(\rm GT)} && \min\{ f(x):~g(x) \le 0,~x\in\mathbb{R}^n\},\\
{(\rm GT_e)} && \min\{ f(x):~g(x) = 0,~x\in\mathbb{R}^n\}.
\end{eqnarray*}

The Euclidean-ball constrained (GT) $(B=I,b=0,c<0)$ is known as the {\it classical} trust-region subproblem (T). Problem (T) plays a great role in each iteration of the trust-region method \cite{C00,Y15} in nonlinear programming. It also has applications in robust optimization \cite{HK18},  constrained linear regression and tensor decomposition \cite{P20}. In the early 1980s, Gay \cite{G81}, Sorensen \cite{S82}, Mor\'{e} and Sorensen \cite{M83} established the necessary and sufficient optimality condition for the global minimizer of (T) even it is nonconvex. Based on this optimality condition, (T) is shown to be polynomially solvable in the early 1990s, see \cite{V90,Y92}. After that, many other  polynomial-time algorithms were proposed to globally solve (T),  for example, see \cite{A17,G99,R97}. Recently, first-order methods were proposed to globally solve large dimensional (T) \cite{B18,H16,H17,W17}. These make sense due to the inherent hidden convexity of (T), see the recent survey \cite{X20} and references therein.

In addition to the applications of its special case (T), (GT) (or $(\rm GT_e)$) not only frequently appears as a subproblem in some algorithm schemes \cite{B09,HN16,WZ20}, but also has direct applications in constrained least squares \cite{G80}, regularized total least squares \cite{B06}, GPS localization and circle fitting problem \cite{B12}.

(GT) and $(\rm GT_e)$ have attractive properties similar to (T). In 1993,  Mor\'{e} \cite{M93} established the necessary and sufficient optimality condition
for the global minimizer of (GT). Under mild assumptions, (GT) and $(\rm GT_e)$ admit strong duality thanks to the fundamental S-lemma \cite{PT07,Y71} and S-lemma with equality \cite{X16}, respectively. Under assumptions such as primal and/or dual Slater conditions, there are many polynomial-time algorithms \cite{A19,F12,M93,P14,S95} and  first-order methods \cite{J20,WK20} for globally solving (GT). Among these assumptions, a notable assumption for (GT) is simultaneously diagonalizable (SD) by congruence for matrices $A$ and $B$, see \cite{J16} and references therein. We refer to \cite{B20,L15} for understanding hidden convexity of (GT) under SD assumption. Efficient algorithms based on SD assumption can be found in \cite{BH14,J19,J18,S18}. Beyond these regular assumptions, (GT) may be either unbounded below or bounded below but the optimal value is unattainable. Necessary and sufficient conditions were established for these two irregular cases by Hsia et al. \cite{H14}.

Local nonglobal minimizers of (T) also have some nice properties.  In 1994, Mart\'{\i}nez \cite{M94} surprisingly proved that (T) has at most one local nonglobal minimizer. The proof is based on univariate reformulations of the standard second-order necessary and sufficient optimality conditions for the local nonglobal minimizer of (T). In 1998, the strict complementarity condition is proved to hold at the local nonglobal minimizer of (T) \cite{LPR98}.
In 2020, Wang and Xia \cite{WX20} proved that, at local nonglobal minimizer of (T), the standard second-order sufficient  optimality condition is surprisingly necessary. This theoretical observation implies that,
by an improved version of generalized eigenvalue-based algorithm \cite{A17,S17}, finding the local nonglobal minimizer of (T) or proving the nonexistence can be done in polynomial time \cite{WX20}. Very recently,
Wang et al. \cite{W21} proved that the local nonglobal minimizer of (T), if exists, has the second smallest objective function value among all KKT points.

Local nonglobal minimizers of (GT) take an irreplaceable role in globally solving the extended generalized trust-region subproblem \cite{TS20}:
\[
\min \left\{f(x):~g(x)\le 0,~h^Tx\le d,~x\in \mathbb{R}^n\right\},\nonumber
\]
where $h\in \mathbb{R}^n$ and $d\in\mathbb{R}$. The unique global optimization solution approach, dating back to \cite{H13}, is to enumerate (candidates of) all local minimizers of (GT) satisfying $h^Tx< d$ and the global minimizer of the reduced version of (GT) in the hyperplane $h^Tx=d$, in the worst case that the linear constraint $h^Tx\le d$ cuts the optimal solution set of (GT).

To our knowledge, there are few investigations on local nonglobal minimizers of (GT) or $(\rm GT_e)$.  In 2020, Taati and Salahi \cite{TS20} extended Mart\'{\i}nez's analysis \cite{M94} to (GT) and established only necessary optimality condition for local nonglobal minimizers of (GT). In the same paper, they extended the eigenvalue-based algorithm \cite{S17} to find {\it candidates} of  local nonglobal minimizers of (GT).  Very recently, Song et al. \cite{S21} proved that there is no local nonglobal minimizer for homogeneous (GT) ({\rm or} $(\rm GT_e)$). They also analyzed local optimality conditions for homogeneous quadratic optimization with two quadratic constraints under a strong assumption, which only covers a special case of homogenized (GT) with a positive definite matrix $B$.
Comparing with (T), the following questions are unknown.
\begin{itemize}
\item[(1)] How many local nonglobal minimizers may (GT) (or $(\rm GT_e)$) have?
\item[(2)] Does the strict complementarity condition hold at any local nonglobal minimizer of (GT)?
\item[(3)] Is there a necessary and sufficient optimality condition for any local nonglobal minimizer of (GT) (or $(\rm GT_e)$)?
\item[(4)] Can we find all local nonglobal minimizers of (GT) (or $(\rm GT_e)$)  in polynomial time?
\end{itemize}

The goal of this paper is to (partially) answer the above four questions. We overestimate the number of local nonglobal minimizers of (GT) (or $(\rm GT_e)$) and show that our bound is tight for two-dimensional case. We can construct an instance of two-dimensional (GT) with two local nonglobal minimizers. In contrast, we prove that the complex-valued (GT) in arbitrary dimension has no local nonglobal minimizer. We provide positive answers to Questions (2)-(4). More precisely, we prove that, at any local nonglobal minimizer of (GT) (or $(\rm GT_e)$), the standard second-order sufficient optimality condition is necessary. As a corollary, finding all local nonglobal minimizers of (GT) (or $(\rm GT_e)$) or proving the nonexistence can be done in polynomial time.

The remainder of the paper is organized as follows.  Section 2 presents  assumptions made on (GT) and $(\rm GT_e)$ and existing characterizations of local and global minimizers. Section 3 proves not only the strict complementarity condition, but also second-order necessary and sufficient optimality condition. The complexity of finding all local nonglobal minimizers of (GT) (or $(\rm GT_e)$) is presented. Section 4 overestimates the number of local nonglobal minimizers of (GT) and $(\rm GT_e)$,  presents an example with more than one local nonglobal minimizers, and proves that (GT) (or $\rm GT_e$) in complex domain has no local nonglobal minimizer. Conclusions are made in Section \ref{sec5}.

\textbf{Notations.} Let $\mathbb{R}^n$ and $\mathbb{C}^n$ be the $n$-dimensional real and complex spaces, respectively. The real and imaginary part of scalars, vectors, and matrices are denoted by $\Re(\cdot)$ and $\Im(\cdot)$, respectively. Denote by $\mathbb{R}^{n \times n}$ ($\mathbb{C}^{n \times n}$) the set of real (complex) square matrices of order $n$. $0$ represents a zero number, an all-zero vector, or an all-zero matrix. For a symmetric matrix $A$, $A\succ(\succeq)0$ denotes that $A$ is positive (semi)definite.
Diag$(a_1,\ldots,a_n)$ returns a diagonal matrix with diagonal components $a_1,\cdots,a_n$. Denote by $[a_1;a_2]$ the column vector composed of vectors $a_1$ and $a_2$. 
For a real number $a$, $|a|$ represents its absolute value. For any smooth vector-valued function $h:\mathbb R\rightarrow \mathbb R^m$ ($m\ge1$), $h'$, $h''$ and $h'''$ denote the first, second and third derivatives of $h$, respectively.

\section{Preliminaries}\label{sec2}
We present in this section some  basic assumptions and existing global and local optimality conditions for $\rm (GT)$ and $\rm (GT_e)$.
\subsection{Assumptions}
We assume that $g(x)$ is nonlinear, since otherwise, (GT)  (or  $\rm (GT_e)$) reduces to a quadratic programming problem, which has no local nonglobal minimizer. Moreover, through a suitable linear transformation if necessary, we can make the following assumption to simplify $g(x)$.
\begin{asmp}\label{as1}
\begin{eqnarray}
B=\left[ {\begin{array}{*{20}{c}}
B_1& \\
&0
\end{array}} \right],~b=\left[ {\begin{array}{*{20}{c}}
0\\
b_2
\end{array}} \right],~c\in \{-1,1,0\},~B_1\in \mathbb{R}^{n_1 \times n_1},~b_2\in \mathbb{R}^{n_2}, \nonumber
\end{eqnarray}
where $B_1$ is symmetric and nonsingular, $n_1$ and $n_2$ are nonnegative integers satisfying $n_1\ge 1$ (so that $g(x)$ is nonlinear) and  $n_1+n_2=n$.
\end{asmp}

The primal Slater conditions for  (GT) and   $\rm (GT_e)$ are presented as follows.
\begin{asmp}\label{as2}
The primal Slater condition holds for $(\rm GT)$, i.e., there exists $\hat x\in \mathbb{R}^{n}$ such that $g(\hat x)<0$.
\end{asmp}
\begin{asmp}\label{as3}
The primal Slater condition holds for $(\rm GT_e)$, i.e., there exist $\hat x, \bar x\in \mathbb{R}^{n}$ such that $g(\hat x)<0<g(\bar x)$.
\end{asmp}
It follows from Assumptions \ref{as2} and \ref{as3} that the feasible regions of (GT) and ${(\rm GT_e)}$ are both nonempty. Moreover, if Assumption \ref{as2} is violated, then there is no local nonglobal minimizer for (GT) (see \cite{TS20}, Proposition 2.3). If Assumption \ref{as3} does not hold, according to  \cite[Lemma 3.1]{M93}, the feasible region of ${(\rm GT_e)}$ is either a linear manifold or a single point. Therefore,  ${(\rm GT_e)}$ has no local nonglobal minimizer.


In   global optimization of (GT) and ${(\rm GT_e)}$, the following standard dual Slater conditions are often used, respectively, see for example{,}  \cite{A19,J19,M93}.
\begin{asmp}\label{as4}
There exists $\hat \lambda \ge0$ such that $A+\hat \lambda B\succ 0$.
\end{asmp}
\begin{asmp}\label{as5}
There exists $\hat \lambda \in \mathbb{R}$ such that $A+\hat \lambda B\succ 0$.
\end{asmp}
Assumption \ref{as5} is slightly weaker than Assumption \ref{as4}.  It was proved in \cite{M93} that under Assumption \ref{as5} (or \ref{as4}), (GT) (or ${(\rm GT_e)}$) has a global minimizer. For the necessary and sufficient condition on the existence of the global minimizer of (GT) (or ${(\rm GT_e)}$), we refer to \cite{H14}.

In local optimality analysis of (GT),
Taati and Salahi \cite{TS20} made Assumptions \ref{as1}, \ref{as2} and \ref{as5}. In particular,
if Assumption \ref{as5} holds but Assumption \ref{as4} fails, Taati and Salahi presented a lower unbounded example with one local nonglobal minimizer, see \cite[Example 1]{TS20}.

Our local optimality analysis in this paper is based on the following further relaxed assumption, which was first introduced for (GT) in \cite{F12}.
\begin{asmp}\label{as6}
There  exist $\mu_1,\mu_2\in \mathbb{R}$ such that $\mu_1A+\mu_2B\succ 0$.
\end{asmp}
Actually, Assumption \ref{as6} is  a well-known sufficient condition for simultaneous diagonalization (via congruence) of $A$ and $B$.

If Assumption \ref{as6} is satisfied but Assumption \ref{as5} fails, (GT) may still have a local nonglobal minimizer even if it is {unbounded below.}
\begin{exam}\label{ex2}
Consider the following three-dimensional case of $\rm (GT)$:
{
\[
\min\{x_1^2+2x_2^2-x_3^2+2x_1-2x_3:~x_1^2-x_2^2+2x_1+2x_3+1\le 0,~x\in\R^3\}.
\]
We can check the lower unboundedness of this instance. We can further verify that  $x^*=[-1,0,0]^T$ is a local nonglobal minimizer of $\rm (GT)$ with the corresponding Lagrangian multiplier $\lambda^* =1$.}
\end{exam}

\subsection{Optimality conditions}

The necessary and sufficient optimality conditions for global minimizers of $\rm (GT)$ and $\rm (GT_e)$ are due to Mor\'{e} \cite{M93}.
\begin{thm}[\cite{M93}, Theorem  3.2]\label{thm:gc}
Under Assumptions \ref{as1} and \ref{as3},
$x^*$ is a global minimizer of   $\rm (GT_e)$, if and only if $g(x^*)=0$ and there exists  $\lambda^*\in \mathbb{R}$ such that
\begin{eqnarray}
&&(A+ \lambda^* B) x^*+a+ \lambda^* b=0,\label{lc1}\\
&&A+\lambda^*B\succeq 0.\label{lc111}
\end{eqnarray}
\end{thm}

\begin{thm}[\cite{M93}, Theorem 3.4]\label{thm:gc2}
Under Assumptions \ref{as1} and \ref{as2},
$x^*$ is a global minimizer of $\rm (GT)$, if and only if $g(x^*)\le 0$, and  \eqref{lc1}-\eqref{lc111} are satisfied for some $\lambda^*\ge 0$ with $\lambda^*= 0$ if $g(x^*)< 0$.
\end{thm}

{
In order to study local minimizers of $\rm (GT)$, we first show that the inequality} constraint is always active at any local nonglobal minimizer of $\rm (GT)$.
\begin{lem}\label{lem:x}
Let  $x^*$ be a local nonglobal minimizer of $\rm (GT)$. Then $g(x^*)=0$.
\end{lem}
\begin{proof}
Suppose, on the contrary, $g(x^*)<0$. The assumptions on
 $x^*$ imply that it is a local minimizer of the unconstrained optimization $\min f(x)$. We have
\[
\nabla f(x^*)=0,~ \nabla^2f(x^*)\succeq 0.
\]
As $f(x)$ is a quadratic function, $x^*$ is a global minimizer of $\min f(x)$. Therefore, $x^*$ is  a  global minimizer of $\rm (GT)$,
 which contradicts the assumption on $x^*$.
\end{proof}

Local optimality conditions require some regular assumptions such as the linear independence constraint qualification (LICQ). Taati and Salahi \cite[Lemma 2.5]{TS20} proved that LICQ always holds at any local nonglobal minimizer of $\rm (GT)$. We extend this result to  $\rm (GT_e)$. Our proof is new and elegant.
\begin{thm}\label{thm:regular}
Under Assumptions  \ref{as1} and \ref{as2} (or \ref{as3}),
if $x^*$ is a local nonglobal minimizer of $\rm (GT)$ (or $\rm (GT_e)$), then LICQ holds at $x^*$.
\end{thm}
\begin{proof}
Suppose on the contrary that LICQ fails at $x^*$, that is,
	\begin{eqnarray}
	\nabla g(x^*)=2(Bx^*+b)=0.\label{eq:fLICQ}
	\end{eqnarray}
By Lemma \ref{lem:x}, $g(x^*)=0$ even for (GT). Then,
it holds that
\[
g(x^*+d)=g(x^*)+\nabla g(x^*)^Td+ d^TBd= d^TBd,~
\forall d\in\mathbb{R}^n.
\]
Therefore, for any $d\in\mathbb{R}^n$ satisfying $d^TBd\le0(=0)$, $x^*+d$ is feasible to $\rm (GT)$ (or $\rm (GT_e)$). Define
\[
h(t)=f(x^*+td)=f(x^*)+2(Ax^*+a)^Tdt+ d^TAdt^2.
\]
It follows from the local optimality of $x^*$ that $h'(0)=0\le h''(0)$, i.e.,
\begin{eqnarray}
		d^TBd\le0(=0)~\Longrightarrow~  (Ax^*+a)^Td=0,~ d^TAd\ge0.\label{eq:con}
	\end{eqnarray}
We partition the remainder of the proof into three cases.
\begin{itemize}	
\item[(a)] $B\succeq0$. By the fact  $g(x^*)=0$ and \eqref{eq:fLICQ}, $g(x)\ge 0$ holds for any $x\in\mathbb{R}^n$. Neither Assumption  \ref{as2} nor Assumption \ref{as3} holds. Consequently, neither $\rm (GT)$ or $\rm (GT_e)$  has a local nonglobal minimizer.
\item[(b)] $B\preceq0$. Similar to (a), we have $g(x)\le 0$, $\forall x\in\mathbb{R}^n$. Then $\rm (GT)$ reduces to the unconstrained  optimization $\min f(x)$ without local nonglobal minimizer.  $\rm (GT_e)$ has no local nonglobal minimizer as Assumption \ref{as3} does not hold.
\item[(c)] $B$ is indefinite.  According to \cite[Lemma 3.10]{P97},
${\rm span}\{v:\ v^TBv=0\}=\mathbb{R}^n$.  Then it follows from \eqref{eq:con}  that $(Ax^*+a)^Td=0$ for all $d\in\mathbb{R}^n$, which implies that
\begin{equation}
Ax^*+a=0.\label{eq:A}
\end{equation}
By applying S-lemma \cite{PT07,Y71} (S-lemma with equality \cite{X16}) to
\[
d^TBd\le0~(d^TBd=0)\Longrightarrow  d^TAd\ge0,
\]
which is given in \eqref{eq:con}, we have
\begin{equation}
\exists\lambda\ge0~(\lambda\in\mathbb{R}):~  A+\lambda B\succeq 0.\label{eq:H}
\end{equation}
Combining \eqref{eq:fLICQ}, \eqref{eq:A}-\eqref{eq:H} with \eqref{lc1}-\eqref{lc111}, we conclude that $x^*$ is a global minimizer of $\rm (GT)$ (or $\rm (GT_e)$) according to Theorem  \ref{thm:gc2} (or \ref{thm:gc}).  This contradicts the assumption on $x^*$.
\end{itemize}
\end{proof}

For completeness, we write down the standard second-order necessary and sufficient optimality conditions for local minimizers of (GT) and $\rm (GT_e)$, respectively.
\begin{thm}\label{thm:lc}
\begin{itemize}
\item[(1)] ({\bf Necessary optimality condition for $\rm (GT)$}) If $x^*$ is a local minimizer of $\rm (GT)$ at which LICQ holds, then $g(x^*)\le0$ and there exists $\lambda^*\ge 0$ such that \eqref{lc1} holds. Moreover, if $g(x^*)=0$ and $\lambda^* >0$, then
\begin{eqnarray}
&&v^T(A+\lambda^* B)v\ge 0,~\forall v\in \mathbb{R}^n~\text{such that}~v^T(Bx^*+b)=0,\label{lc2}
\end{eqnarray}
and if $g(x^*)=0$ and $\lambda^* =0$, it holds that
\begin{eqnarray}
&&v^T(A+\lambda^* B)v\ge 0,~\forall v\in \mathbb{R}^n~\text{such that}~v^T(Bx^*+b)\le0.\label{lc4}
\end{eqnarray}
\item[(2)] ({\bf Sufficient optimality condition for $\rm (GT)$}) Suppose that $g(x^*)=0$ and there exists $\lambda^* > 0$ satisfying \eqref{lc1} and
\begin{eqnarray}
&&v^T(A+\lambda^* B)v> 0,~\forall v\in \mathbb{R}^n~\text{such that}~v^T(Bx^*+b)=0,~v\neq 0.\label{lc3}
\end{eqnarray}
Then $x^*$ is a strict local minimizer of $\rm (GT)$.
\item[(3)] ({\bf Necessary optimality condition for $\rm (GT_e)$})  If $x^*$ is a local minimizer of $\rm (GT_e)$  at which  LICQ holds, then $g(x^*)=0$ and there exists $\lambda^* \in \mathbb{R}$ such that \eqref{lc1} and \eqref{lc2} hold.
\item[(4)] ({\bf Sufficient optimality condition for $\rm (GT_e)$}) Suppose that $g(x^*)=0$ and there exists $\lambda^* \in \mathbb{R}$ satisfying \eqref{lc1} and
\eqref{lc3}. Then $x^*$ is a strict local minimizer of $\rm (GT_e)$.
\end{itemize}
\end{thm}


\section{Local optimality condition for local nonglobal minimizer}\label{sec3}
In this section, we first show that the strict complementary condition holds at any local nonglobal minimizer  of (GT), and then prove that the standard second-order sufficient optimality condition  for any local minimizer  of (GT) (or $\rm (GT_e)$) is surprisingly necessary.

\begin{thm}\label{thm:ng0}
Under Assumptions \ref{as1} and \ref{as2},
if $x^*$ is a local nonglobal minimizer of $\rm (GT)$, then the strict complementary condition holds at $x^*$, i.e., the corresponding Lagrangian multiplier is positive.
\end{thm}
\begin{proof}
By Lemma \ref{lem:x}, $g(x^*)=0$.
According to Theorem \ref{thm:regular}, LICQ holds at $x^*$.
Let $\lambda^* $ be the Lagrangian multiplier corresponding to $x^*$. According to Theorem \ref{thm:lc}, \eqref{lc1} holds. Now it is sufficient to prove $\lambda^* > 0$. Suppose, on the contrary, it holds that $\lambda^* =0$. By \eqref{lc4}, we have
\[
v^TAv\ge0,~\forall v\in \mathbb{R}^n ~\text {such that}~ v^T (Bx^*+b)\le 0,
\]
which is equivalent to
\[
v^TAv\ge0,~\forall v\in \mathbb{R}^n,
\]
since $v^TAv=(-v)^TA(-v)$. Therefore,  we obtain $A\succeq 0$. Now  \eqref{lc1}-\eqref{lc111} hold with $\lambda^* =0$. According to Theorem \ref{thm:gc2}, $x^*$ is a global minimizer of (GT). This contradicts the assumption that $x^*$ is a local nonglobal minimizer.
\end{proof}

For local nonglobal minimizer of $\rm (GT)$ (or $\rm (GT_e)$) $x^*$ and its Lagrangian multiplier $\lambda^*$, according to
Theorem \ref{thm:gc2} (or Theorem \ref{thm:gc}), $A+\lambda^* B$ has at least one negative eigenvalue. On the other hand, as a direct corollary of  Theorems \ref{thm:regular}, \ref{thm:lc} and \ref{thm:ng0}, $A+\lambda^* B$ has at most one negative eigenvalue. Consequently, we have established the following necessary optimality condition for any local nonglobal minimizer of $\rm (GT)$ (or $\rm (GT_e)$).
\begin{cor}[\cite{TS20}, Lemma 3.1]\label{prop7}
Let $x^*$ be a local nonglobal minimizer of $\rm (GT)$ (or $\rm (GT_e)$) and $\lambda^* \ge 0$ (or $\lambda^* \in \mathbb{R}$) be the corresponding Lagrangian multiplier.
Under Assumptions \ref{as1} and \ref{as2} (or Assumptions \ref{as1} and \ref{as3}), $A+\lambda^* B$ has exactly one negative eigenvalue.
\end{cor}

As a main result, we prove that the standard second-order sufficient optimality condition for any local nonglobal minimizer of $\rm (GT_e)$ is  necessary.
\begin{thm}\label{thm:ng1}
Suppose Assumptions \ref{as1}, \ref{as3} and \ref{as6} hold, and $x^*$ is not a global minimizer of $\rm (GT_e)$. Then
$x^*$ is a local minimizer of $\rm (GT_e)$ if and only if $g(x^*)=0$ and there exists a $\lambda^* \in \mathbb{R}$ satisfying \eqref{lc1} and \eqref{lc3}.  
\end{thm}
\begin{proof}
Suppose that $x^*$ is a local nonglobal minimizer of $\rm (GT_e)$. Let $\lambda^* \in \mathbb{R}$ be the corresponding Lagrangian multiplier. Then $g(x^*)=0$, \eqref{lc1} and \eqref{lc2} hold by Theorem \ref{thm:lc}. Then $A+\lambda^*  B\nsucceq0$ by Theorem \ref{thm:gc}. Let $G=A+\lambda^*  B$. We rewrite
the corresponding Lagrangian function as follows:
\begin{eqnarray}
L_1(x)=f(x)+\lambda^*  g(x)&=&x^TGx+2(a+\lambda^*  b)^Tx+ \lambda^*  c\nonumber\\
&=&x^TGx-2x^{*T}Gx+ \lambda^*  c,\label{lx}
\end{eqnarray}
where the second equality \eqref{lx} follows from \eqref{lc1}.

It is sufficient to prove that \eqref{lc3} holds. Suppose, on the contrary, there exists a vector $\bar v\neq 0$ such that
\begin{eqnarray}
&& \bar v^TG\bar v=0,\label{v1} \\
&& \bar v^T(Bx^*+b)=0.\label{v2}
\end{eqnarray}
We first claim that
\begin{eqnarray}
&&\bar v^TB\bar v\neq 0. \label{vb}
\end{eqnarray}
If this is not true, we have $\bar v^TB\bar v= 0$. By \eqref{v1}, we obtain  $\bar v^TA\bar v= 0$. Thus, for any $\mu_1,\mu_2 \in \mathbb{R}$, we have
\[
\bar v^T(\mu_1A+\mu_2 B)\bar v= 0,
\]
which contradicts Assumption \ref{as6}.

We partition the remainder of the proof into two cases.

{\bf Case (a)}. Assume
\begin{eqnarray}
&& G\bar v= 0.\label{gv1}
\end{eqnarray}
According to Corollary \ref{prop7}, $G$ has exactly one negative eigenvalue. Let $w\in \mathbb{R}^n$ be a nonzero eigenvector of $G$ corresponding to the unique negative eigenvalue. Then we must have
\[
w^T(Bx^*+b)\neq 0,
\]
since otherwise, according to \eqref{lc2}, it holds that $w^TGw\ge 0$,  which contradicts the definition of $w$.

Define the function $h:\mathbb{R}^2\mapsto \mathbb{R}$ as
\begin{eqnarray}
h(s,t)&:=&g(x^*+t\bar v+sw)\nonumber \\
&=&(x^*+t\bar v+sw)^TB(x^*+t\bar v+sw)
+2b^T(x^*+t\bar v+sw)+c\nonumber \\
&=&\bar v^TB\bar v\cdot t^2+w^TBw\cdot s^2+2w^T(B(x^*+t\bar v)+b)\cdot s, \label{h}
\end{eqnarray}
where \eqref{h} follows from the fact $g(x^*)=0$ and \eqref{v2}. We can verify that
\[
h(0,0)=0,~\frac{\partial h}{\partial s}(0,0)=2w^T(Bx^*+b)\neq 0.
\]
According to the implicit function theorem in multi-variable calculus, there are two open intervals $I_1,I_2\subset \mathbb{R}$ with $0\in I_1\bigcap I_2$ and a continuous function $s(t):I_1\mapsto I_2$ such that $s(0)=0$. Moreover, for any $t\in I_1$, the point $s(t)\in I_2$ is a unique point satisfying
\begin{eqnarray}
&&h(s(t),t)=0.\label{st1}
\end{eqnarray}
By \eqref{vb}, \eqref{h} and \eqref{st1}, it holds that
\begin{eqnarray}
&&s(t)\neq 0,~\forall t\neq0.\label{st2}
\end{eqnarray}
Define $x(t):=x^*+t\bar v+s(t)w$. By the definition of $h(s,t)$ and \eqref{st1}, we have $g(x(t))=0$. Therefore, for any sufficiently small $|t|$ and $t\neq 0$, $x(t)$ is a feasible point in the neighborhood of $x^*$. Moreover, according to \eqref{lx}, \eqref{gv1} and \eqref{st2}, we can verify that
\[
f(x(t))=L_1(x(t))=w^TGw\cdot s(t)^2+L_1(x^*)<L_1(x^*)=f(x^*),
\]
which contradicts the assumption that $x^*$ is a local minimizer of $\rm (GT_e)$.

{\bf Case (b)}. Assume
\begin{eqnarray}
&& G\bar v\neq 0.\label{gv2}
\end{eqnarray}
We further consider three subcases
with different values of $c\in\{-1,1,0\}$.

{\bf Subcase (b1)}. $c=-1$. For any real number $t$, we can verify that
\begin{eqnarray}
&&(x^*+t\bar v)^TB(x^*+t\bar v)+2b^T(x^*+t\bar v)\nonumber\\
&=& x^{*T}Bx^*+2b^Tx^*+2\bar v^T(Bx^*+b)\cdot t
+\bar v^TB\bar v\cdot t^2 \label{xtv1}\\
&=&1+\bar v^TB\bar v\cdot t^2, \label{xtv2}
\end{eqnarray}
where \eqref{xtv2} follows from the fact $g(x^*)=0$, \eqref{v2} and $c=-1$.

Let $x^*=[x^*_1;x^*_2]$ and $\bar v=[\bar v_1;\bar v_2]$, where $x^*_1,\bar v_1\in \mathbb{R}^{n_1}$ and $x^*_2,\bar v_2\in \mathbb{R}^{n_2}$. We define $x(t)=[x_1(t);x_2(t)]$ as
\begin{eqnarray}
&&  x_1(t)=
  \frac{x^*_1+t\bar v_1}{\sqrt{1+\bar v^TB\bar v\cdot t^2}},
  ~ x_2(t)=   \frac{x^*_2+t\bar v_2}{1+\bar v^TB\bar v\cdot t^2}.\label{x1x2}
\end{eqnarray}
Then for any sufficiently small $|t|$, the parametric curve $x(t)$ is well defined. In the remainder of the proof, we always assume that $|t|$ is sufficiently small. According to \eqref{xtv1}-\eqref{x1x2}, we have $x(0)=x^*$ and $g(x(t))=0$. Thus, for any $t$, $x(t)$ is a feasible solution of $\rm (GT_e)$.
Consequently, $0$ is a local minimizer of the univariate function defined by
\[
F(t):=f(x(t))=L_1(x(t))=x(t)^TGx(t)-2x^{*T}Gx(t)-\lambda^*.
\]
Notice that both $x(t)$ and $F(t)$ are sufficiently smooth.  By elementary analysis, we can verify that
\begin{eqnarray}
  x'(0) &=& \bar v,\label{xt1} \\
  x''(0)  &=& -\bar v^TB\bar v\cdot [x^*_1;2x^*_2],\label{xt2}\\
  F'(0) &=&0, \nonumber \\
  F''(0)&=&2x'(0)^TGx'(0)=2\bar v^TG\bar v=0, \label{Fv1} \\
  F'''(0) &=&6x'(0)^TGx''(0)=6\bar v^TGx''(0),\label{Fv2}
\end{eqnarray}
where the item $x'(0)$ in \eqref{Fv1}-\eqref{Fv2} is replaced by $\bar v$ due to \eqref{xt1}, and
the last equality in \eqref{Fv1} follows from \eqref{v1}.

Let $W\in \mathbb{R}^{n\times (n-1)}$ be a full column-rank matrix whose columns form a basis of the $(n-1)$-dimensional subspace given by
\begin{eqnarray}
&& \{v\in\mathbb{R}^{n}:~v^T(Bx^*+b)=0\}.\label{vw}
\end{eqnarray}
That is, $W^T(Bx^*+b)=0$. By \eqref{lc2} in Theorem \ref{thm:lc}, we have
\begin{eqnarray}
&& W^TGW=W^T(A+\lambda^*  B)W\succeq 0.\label{u2}
\end{eqnarray}
According to \eqref{v2} and
$\bar v\neq 0$, there exists a nonzero $\bar u\in \mathbb{R}^{n-1}$ such that
\begin{eqnarray}
&&\bar v=W\bar u.\label{u0}
\end{eqnarray}
Substituting \eqref{u0} into \eqref{v1} yields that
\begin{eqnarray}
&& \bar u^TW^TGW\bar u=0.\label{u1}
\end{eqnarray}
Then it follows from \eqref{u2}-\eqref{u1} that
\begin{eqnarray}
0=W^TGW\bar u=W^TG\bar v.\label{u3}
\end{eqnarray}
Since $W^T(Bx^*+b)=0$ and $W$ is of full column-rank, \eqref{u3} implies  that
\begin{eqnarray}
G\bar v=\gamma (Bx^*+b) \label{u4}
\end{eqnarray}
with $\gamma\neq 0$ due to \eqref{gv2}. Substituting \eqref{xt2} and \eqref{u4}  into \eqref{Fv2} leads to
\begin{eqnarray}
&& F'''(0)=-6\gamma \cdot \bar v^TB\bar v \cdot (x^{*T}Bx^*+2b^Tx^*)=-6\gamma \cdot \bar v^TB\bar v\neq 0,\nonumber
\end{eqnarray}
where the second equality holds as $g(x^*)=0$ and $c=-1$, and the last inequality follows from $\gamma\neq 0$ and \eqref{vb}. Therefore, $t=0$ is not a local minimizer of $F(t)$. This  contradicts the fact that
 $x^*$ is a local minimizer of $\rm (GT_e)$.

{\bf Subcase (b2)}. $c=1$. By \eqref{xtv1}, we have
\[
(x^*+t\bar v)^TB(x^*+t\bar v)+2b^T(x^*+t\bar v)
=-1+\bar v^TB\bar v\cdot t^2.
\]
Then we define $x(t)=[x_1(t);x_2(t)]$ by
\begin{eqnarray}
&&  x_1(t)=
  \frac{x^*_1+t\bar v_1}{\sqrt{1-\bar v^TB\bar v\cdot t^2}},
  ~ x_2(t)=   \frac{x^*_2+t\bar v_2}{1-\bar v^TB\bar v\cdot t^2},\nonumber
\end{eqnarray}
as a replacement of the definition $x(t)$ in \eqref{x1x2}.  Based on a proof similar to that of Subcase (b1), we can obtain a contradiction for Subcase (b2).

{\bf Subcase (b3)}. $c=0$. We first consider the case that $n_2\ge 1$ and $b_2\neq 0$. For any feasible point of $\rm (GT_e)$ with the partition $x=[x_1;x_2]$ where $x_1\in \mathbb{R}^{n_1}$ and $x_2\in \mathbb{R}^{n_2}$, we have
\[
g(x)=x_1^TB_1x_1+2b_2^Tx_2+0=x_1^TB_1x_1+2b_2^T\left(x_2+\frac{b_2}{2\|b_2\|^2}\right)-1.
\]
Then by a linear transformation, we are back to Subcase (b1) with respect to the new variables
\[
\hat x:=\left[x_1; x_2+\frac{b_2}{2\|b_2\|^2}\right].
\]

The remainder case is either $n_2\ge 1$ and $b_2=0$ or $n_2=0$. In this case, we always have $b=0$. $\rm (GT_e)$ can be equivalently homogenized to the following equality-constrained quadratic optimization:
\[
{\rm (QQ_2)}~~ \min \left\{q_0(y):~q_1(y)= 1,~q_2(y)=1,~y\in \mathbb{R}^{n+1}\right\},
\]
where $q_i(y)=y^TA_iy$ for $i=0,1,2$, and
\begin{equation}
A_0=\left[ {\begin{array}{*{20}{c}}
A& a\\
a^T&0
\end{array}} \right],~
A_1=\left[ {\begin{array}{*{20}{c}}
0& 0 \\
0 & 1
\end{array}} \right],~
A_2=\left[ {\begin{array}{*{20}{c}}
B&0 \\
0& 1
\end{array}} \right].  \label{A123}
\end{equation}
Clearly, $x^*$ is a local nonglobal minimizer of $\rm (GT_e)$, if and only if $y^*:=[x^*;1]$ is a local nonglobal minimizer of ${\rm (QQ_2)}$.  The corresponding Lagrangian function is given by
\begin{eqnarray}
L_2(y)&=&q_0(y)+\lambda_1(q_1(y)-1)
+\lambda_2(q_2(y)-1)\nonumber\\
&=&y^T(A_0+\lambda_1 A_1+\lambda_2 A_2)y-\lambda_1-\lambda_2. \label{ly11}
\end{eqnarray}
Since LICQ clearly holds at $y^*$, by the first-order optimality condition, it holds that
\begin{eqnarray}
&&(A+\lambda_2 B)x^*+a =0,\label{mu0}\\
&&a^Tx^*+\lambda_1+\lambda_2 =0.\label{mu1}
\end{eqnarray}
Substituting \eqref{lc1} into \eqref{mu0} yields that
\[
(\lambda_2-\lambda^*)Bx^*=0.
\]
By Theorem \ref{thm:regular}, $\nabla g(x^*)=2Bx^*\neq 0$. Therefore, we have $\lambda_2=\lambda^*$. It follows from  \eqref{lc1} and $g(x^*)=0$ that
\begin{eqnarray}
&&a^Tx^*=-x^{*T}(A+\lambda^* B)x^*=-x^{*T}Ax^*.\label{mu2}
\end{eqnarray}
Then, by \eqref{mu1}-\eqref{mu2} and $\lambda_2=\lambda^*$, we have $\lambda_1=x^{*T}Ax^*-\lambda^*$. Therefore, by introducing $H=A_0+(x^{*T}Ax^*-\lambda^* )A_1+\lambda^* A_2$, \eqref{ly11}  reduces to
\begin{eqnarray}
&&L_2(y)=y^THy-x^{*T}Ax^*.\nonumber
\end{eqnarray}
Define $\bar w:=[\bar v;0]\in\R^{n+1}$, where $\bar v$ satisfies \eqref{v1}-\eqref{v2} and \eqref{u4} with $\gamma \neq 0$. We can verify that
\begin{eqnarray}
&&Hy^*=0,\label{wy0}\\
&&\bar w^TH\bar w=\bar v^TG\bar v=0,\label{wy1}\\
&&\bar w^TA_1y^*=0,~\bar w^TA_2y^*=0,\label{wy2}\\
&&a^T\bar v=-x^{*T}G \bar v=-\gamma x^{*T}Bx^*=0,\label{wy3}
\end{eqnarray}
where \eqref{wy0} follows from \eqref{lc1} and \eqref{mu2}, \eqref{wy1}-\eqref{wy2} hold due to \eqref{v1}-\eqref{v2}, and the three equalities in \eqref{wy3} are implied from \eqref{lc1}, \eqref{u4} and the fact $g(x^*)=0$, respectively. Then,
by \eqref{u4} and \eqref{wy3}, we have
\begin{eqnarray}
&&H\bar w=[G\bar v;a^T\bar v]=[\gamma Bx^*;0].\label{wy4}
\end{eqnarray}
The following proof is extended from the proof for Case (c) in Theorem 4.2 in \cite{S21}.
Actually,  \cite[Theorem 4.2]{S21} states that the standard second-order sufficient optimality condition for any strict local nonglobal minimizer of $\rm (QQ_2)$ is necessary under the assumption  $\mu_1A_1+\mu_2A_2\succ 0$ for some $\mu_1,\mu_2\in\mathbb{R}$. However, according to the definitions \eqref{A123}, this assumption fails to hold for our case  except that $B\succ 0$.


Under Assumption \ref{as6}, $A$ and $B$ are simultaneously diagonalizable by congruence. Through a suitable linear transformation, we can directly assume that both $A$ and $B$ are diagonal matrices. Let
\begin{equation}
A={\rm Diag}(\alpha_1,\cdots,\alpha_n),~B={\rm Diag}(\beta_1,\cdots,\beta_{n_1},0,\cdots,0),\label{AB}
\end{equation}
where $n_1\ge 1$ and $\beta_1,\cdots,\beta_{n_1}\neq 0$. Since $\nabla g(x^*)=2Bx^*\neq 0$ due to Theorem \ref{thm:regular},
without loss of generality, we can assume that the first component of $y^*$ (the same as the first component of $x^*$) is nonzero, i.e., $y^*_1\neq 0$. We further assume $y^*_1> 0$, as the negative case can be proved similarly.

Let $\bar{I}={\rm Diag}(0,1,\cdots,1)$ and $e_1=[1,0,\cdots,0]^T$ and define
\begin{eqnarray}
&&  z(t)=\bar{I}(y^*+t\bar w),
  ~ y(t)= z(t)+\sqrt{\frac{1-z(t)^TA_2z(t)}{\beta_1}}\cdot e_1.\nonumber
\end{eqnarray}
Notice that
\begin{eqnarray}
&&z(0)=[0,y^*_2,\cdots,y^*_n,1]^T,\nonumber\\
&&\frac{1-z(0)^TA_2z(0)}{\beta_1}=y_1^{*2}>0,\label{yt0}\\
&&y(0)=z(0)+\sqrt{\frac{1-z(0)^TA_2z(0)}{\beta_1}} e_1=\bar{I}y^*+y^*_1e_1=y^*.\label{yt1}
\end{eqnarray}
It follows from \eqref{yt0} that $y(t)$ is well defined for sufficiently small $|t|$. Moreover, we can verify that $y(t)$ is a feasible point of  $\rm (QQ_2)$  by the following equalities:
\begin{eqnarray}
&&y(t)^TA_1y(t)=1,\nonumber\\
&&y(t)^TA_2y(t)=z(t)^TA_2z(t)
+\beta_1\left(\sqrt{\frac{1-z(t)^TA_2z(t)}{\beta_1}}\right)^2=1.\nonumber
\end{eqnarray}
According to \eqref{v2}, we can verify that
\begin{eqnarray}
&&z'(0)=\bar I\bar w,~z''(0)=0,\nonumber\\
&&y'(0)=\bar w,~y''(0)=-\frac{\bar v^TB\bar v}{\beta_1y^*_1}e_1.\label{yt2}
\end{eqnarray}
For any sufficiently small $|t|$, we define
\[
\phi(t):=q_0(y(t))=L_2(y(t))=y(t)^THy(t)-x^{*T}Ax^*.
\]
Under the assumption that $y^*$ is a local minimizer of $\rm (QQ_2)$, $t=0$ is a local minimizer of $\phi(t)$. Elementary analysis shows that
\begin{eqnarray}
\phi'(0)&=&2y(0)^THy'(0),\nonumber\\
\phi''(0)&=&2y(0)^THy''(0)+2y'(0)^THy'(0),\nonumber\\
\phi'''(0)&=&2y(0)^THy'''(0)+6y'(0)^THy''(0).\nonumber
\end{eqnarray}
By \eqref{wy0}-\eqref{wy1}, \eqref{wy4} and \eqref{yt1}-\eqref{yt2}, we can simplify the above three equalities as
\begin{eqnarray}
\phi'(0)&=&0,\label{phi:1}\\
\phi''(0)&=&2\bar w^TH\bar w=0,\label{phi:2}\\
\phi'''(0)&=&-6\frac{\bar v^TB\bar v}{\beta_1y_1}
e_1^TH\bar w=-6\gamma \bar v^TB\bar v.\nonumber
\end{eqnarray}
According to the fact $\gamma\neq0$ and \eqref{vb}, we have $\phi'''(0)\neq 0$, which, together with \eqref{phi:1}-\eqref{phi:2}, implies that $t=0$ is not a local minimizer of $\phi(t)$. This contradiction   completes the proof.
\end{proof}

Based on Theorem \ref{thm:ng0} and the proof of Theorem \ref{thm:ng1}, we can show the necessity of the standard second-order sufficient optimality condition for any local nonglobal minimizer of (GT).
\begin{thm}\label{thm:ng2}
Under Assumptions \ref{as1}, \ref{as2} and \ref{as6},
$x^*$ is a local nonglobal minimizer of $\rm (GT)$ if and only if $g(x^*)=0$ holds, and there exists a $\lambda^* > 0$ satisfying \eqref{lc1}, \eqref{lc3} and $A+\lambda^* B \nsucceq 0$.
\end{thm}

Since the standard second-order sufficient optimality condition implies strictness of any local nonglobal minimizer{,
according} to Theorems \ref{thm:ng1} and \ref{thm:ng2}, we have the following result.
\begin{cor}\label{cor:ng}
Under Assumptions \ref{as1}, \ref{as2} (or \ref{as3}) and \ref{as6},
any local nonglobal minimizer of $\rm (GT)$ (or $\rm (GT_e)$) is isolated.
\end{cor}

Taati and Salahi \cite{TS20} extended the eigenvalue-based algorithm \cite{S17} to find at most $2n+1$ candidates of  local nonglobal minimizers of (GT) (or $\rm (GT_e)$). The worst-case complexity is  $O(n^4)$. With the help of Theorem \ref{thm:ng1} (or \ref{thm:ng2}), {now} one can check whether a candidate solution is a local nonglobal minimizer of (GT) (or $\rm (GT_e)$) in $O(n^3)$ time. As a conclusion, we have the following result.
\begin{cor}\label{cor:ng2}
All local nonglobal minimizers of $\rm (GT)$ (or $\rm (GT_e)$) can be found in $O(n^4)$ time.
\end{cor}

\section{Number of local nonglobal minimizers}\label{subsec4}
In this section, we first overestimate the number of local nonglobal minimizers of $\rm (GT)$ (or $\rm (GT_e)$). Then,
by an example, we show that $\rm (GT)$ (or $\rm (GT_e)$) may have two local nonglobal minimizers. In contrast, we can prove that $\rm (GT)$ (or $\rm (GT_e)$) in complex domain has no local nonglobal minimizer. It demonstrates that real-valued optimization problem is more difficult to solve than its complex version.

\subsection{Real case}

Based on the above results, we can now overestimate the number of  local nonglobal minimizers of $\rm (GT)$ (or $\rm (GT_e)$).
\begin{thm}\label{thm:num}
Under Assumptions \ref{as1}, \ref{as2} (or \ref{as3}) and \ref{as6}, the number of  local nonglobal minimizers of $\rm (GT)$ (or $\rm (GT_e)$) is at most $\min\{n_1+1,n\}$.
\end{thm}
\begin{proof}
Let $x^*$ be a local nonglobal minimizer of $\rm (GT)$ (or $\rm (GT_e)$) and $\lambda^* \ge 0$ (or $\lambda^* \in \mathbb{R}$) be the corresponding Lagrangian multiplier.
Under Assumptions \ref{as1}, \ref{as2} (or \ref{as3}) and \ref{as6},
according to Corollary \ref{prop7}, Theorems \ref{thm:ng1} and \ref{thm:ng2},
we have that $A+\lambda^* B$ is nonsingular.
It then follows from \eqref{lc1} that
\begin{equation}
x^*=-(A+ \lambda^* B)^{-1}(a+ \lambda^* b). \label{xstar}
\end{equation}
By substituting \eqref{xstar} into $g(x^*)=0$, we can see that $\lambda^*$ is a zero point of
the following  secular function
\[
 \varphi(\lambda)=
(a+ \lambda b)^T(A+ \lambda  B)^{-1}B(A+ \lambda B)^{-1}(a+ \lambda b)-2b^T(A+ \lambda B)^{-1}(a+ \lambda b)+c.
\]
According to \eqref{AB} under Assumption \ref{as6}, we can rewrite $ \varphi(\lambda)$ as
\[
 \varphi(\lambda)=
\sum_{i=1}^{n_1}\frac{\beta_i(a_i+ \lambda b_i)^2}{(\alpha_i+\lambda  \beta_i)^2}+\sum_{i=n_1+1}^{n_1+n_2}\frac{\beta_i(a_i+ \lambda b_i)^2}{\alpha_i^2}
-2\sum_{i=n_1+1}^{n_1+n_2}\frac{b_i(a_i+ \lambda b_i)}{ \alpha_i }+c.
\]
Therefore, if $n_1<n$, the fraction polynomial $\varphi(\lambda)$ has at most $2(n_1+1)$ zero points, and otherwise $n_1=n$, then $\varphi(\lambda)=0$ has at most $2n_1$ roots. As a summary, $\varphi(\lambda)$ has at most $2\min\{n_1+1,n\}$ zero points.

On the other hand, based on the proof of Theorem 3.7 in \cite{TS20}, one can observe that \eqref{lc3} (which holds due to Theorems \ref{thm:ng1} and \ref{thm:ng2}) implies
\begin{equation}
\varphi'(\lambda^*)> 0.\label{sign}
\end{equation}
Since $\varphi(\lambda)$ is a univariate fraction polynomial and hence $\varphi'(\lambda_1)\varphi'(\lambda_2)\le 0$ for any two adjacent zero points of $\varphi(\lambda)$ denoted by  $\lambda_1$ and $\lambda_2$, half of the $2\min\{n_1+1,n\}$ zero points
 satisfy \eqref{sign}. This completes the proof.
\end{proof}

The next example shows that the upper bound given in Theorem \ref{thm:num} is tight at least for two-dimensional $\rm (GT)$ (or $\rm (GT_e)$).
\begin{exam}\label{ex3}
Consider an example of   $\rm (GT_e)$ with $n=2$:
\begin{equation}
\min \{  y^2+z^2+12y+8z:~ yz=1,~(y,z)\in\mathbb{R}^2\}. \label{yz}
\end{equation}
By substituting $z=1/y$ into the objective function, we obtain the following equivalent univariate unconstrained optimization:
\begin{equation}
 \min \left\{  y^2+\frac{1}{y^2}+12y+\frac{8}{y}:~y\in\mathbb{R}\right\}.
\label{y}
\end{equation}
As plotted in Fig. \ref{fig:1}, the problem \eqref{y} is upper unbounded and has four critical points including two local nonglobal minimizers.
\begin{figure}[!htb]
\begin{center}
 \centering
  \includegraphics[width=0.5\textwidth]{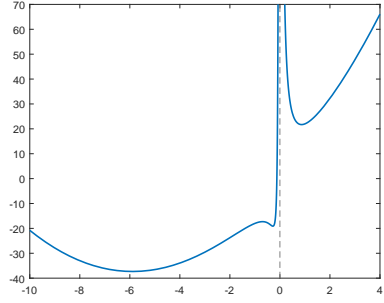}
  \caption{\label{fig:1} Variation of the objective function of  \eqref{y}.}
\end{center}
\end{figure}
We modify \eqref{yz} to the following two-dimensional example of $\rm (GT)$:
\begin{equation}
\min \{ y^2+z^2+12y+8z:~ yz-1\le 0,~(y,z)\in\mathbb{R}^2\}.\label{yz2}
\end{equation}
Problem \eqref{yz2} has the same two local nonglobal minimizers as \eqref{yz}.
\end{exam}
Example \ref{ex3} also implies the following observation.
\begin{prop}\label{prop:ng1}
$\rm (GT)$ (or $\rm (GT_e)$)  may have more than one local nonglobal minimizer.
\end{prop}

\subsection{Complex case}
Complex-valued optimization problems naturally arise in many applications including estimation problems in Fourier domain, signal constellations and narrow-band array processing. The subproblem (34) in \cite{Yu16} for designing the hybrid precoder in millimeter wave MIMO system is exactly a complex-valued trust-region subproblem. In this subsection,
we study the generalized trust-region subproblem in the complex domain:
\begin{eqnarray*}
(\rm GT^{c})&& \min\{z^HAz+2\Re(a^Hz):~
z^HBz+2\Re(b^Hz)+c\le 0,~z\in \mathbb{C}^n\}, \\
(\rm GT_e^{c})&&
\min\{z^HAz+2\Re(a^Hz):~ z^HBz+2\Re(b^Hz)+c=0,~z\in \mathbb{C}^n\},
\end{eqnarray*}
where $A,B\in \mathbb{C}^{n \times n}$ are Hermitian matrices, $a,b\in \mathbb{C}^{n}$ and $c\in \mathbb{R}$. Notice that
with the setting $x={\Re(z)}$ and $y={\Im(z)}$, we can verify that
\begin{eqnarray*}
z^HAz+2\Re(a^Hz)&=&x^T\Re(A)x+y^T\Re(A)y+2{\Re(a)}^Tx+2{\Im(a)}^Ty\\
&-&x^T\Im(A)y+y^T\Im(A)x,
\end{eqnarray*}
where $\Im(A)\in \mathbb{R}^{n\times n}$ is a skew-symmetric matrix, i.e., $\Im(A)^T=-\Im(A)$. Then $(\rm GT^{c})$ and $(\rm GT_e^{c})$ can be rewritten as  the following problems in $\mathbb{R}^{2n}$:
\begin{eqnarray*}
(\rm GT^{r})&& \min\{F(w):~
G(w)\le 0,~w\in \mathbb{R}^{2n}\}, \\
(\rm GT_e^{r})&&
\min\{F(w):~ G(w)=0,~w\in \mathbb{R}^{2n}\},
\end{eqnarray*}
respectively, where
\begin{eqnarray*}
&&F(w)=w^T
	\bmatrix
	\Re(A) & -\Im(A)\\
	\Im(A) & \Re(A)
	\endbmatrix
	w+ 2{(\Re(a), \Im(a))}^Tw, \\
&&G(w)=w^T
	\bmatrix
	\Re(B) & -\Im(B)\\
	\Im(B) & \Re(B)
	\endbmatrix
	w+ 2{(\Re(b), \Im(b))}^Tw+c.
\end{eqnarray*}
The above linear transformation establishes the following equivalence.
\begin{lem}\label{le:2nGT}
$\bar z$ is a local/global minimizer of $(\rm GT^{c})$ (or $(\rm GT_e^{c})$) if and only if $(\Re(\bar z), \Im(\bar z))$ is that of $(\rm GT^{r})$ (or $(\rm GT_e^{r})$).
\end{lem}

Different from real-valued generalized trust-region subproblem, the complex version has no local nonglobal minimizer.

\begin{thm}\label{comlex}
Under complex versions of Assumptions \ref{as1}, \ref{as2} (or \ref{as3}) and \ref{as6}, neither $(\rm GT^{c})$ nor $(\rm GT_e^{c})$ has a local nonglobal minimizer.	
\end{thm}
\begin{proof}
Under complex versions of Assumptions \ref{as1}, \ref{as2} (or \ref{as3}) and \ref{as6}, we can first verify that Assumptions \ref{as1}, \ref{as2} (or \ref{as3}) and \ref{as6} are satisfied for $(\rm GT^{r})$ (or $(\rm GT_e^{r})$).

Notice that the Hessian matrix of the Lagrangian function of  $(\rm GT^{r})$ (or $(\rm GT_e^{r})$) is given by
\begin{eqnarray*}
H(\lambda):=
\bmatrix
\Re(A)&-\Im(A)\\
\Im(A)&\Re(A)
\endbmatrix
+\lambda
\bmatrix
\Re(B)&-\Im(B)\\
\Im(B)&\Re(B)
\endbmatrix.\label{eq:H_lambda}
\end{eqnarray*}
For any fixed $\lambda\in\mathbb{R}$, let $\mu$ be an eigenvalue of $H(\lambda)$ and $[v_1; v_2] (\neq 0)$ be the corresponding eigenvector, then $[v_1; v_2]^T[-v_2; v_1]=0$ and
\[
H(\lambda)[v_1; v_2]=\mu [v_1; v_2] \Longleftrightarrow H(\lambda)[-v_2; v_1]=\mu [-v_2; v_1].
\]
Thus, for any $\lambda\in\mathbb{R}$, the multiplicity of each eigenvalue of $H(\lambda)$ is greater than or equal to two. 
Therefore, according to Corollary \ref{prop7},  $(\rm GT^{r})$ (or $(\rm GT_e^{r})$) has no local nonglobal minimizer.  By Lemma \ref{le:2nGT},  there is no local nonglobal minimizer for $(\rm GT^{c})$ $\rm (or \ (GT_e^{c}))$.
\end{proof}

As an application of Theorem \ref{comlex},
the extended generalized trust-region subproblem in complex domain
\[
{\min} \{z^HAz+2\Re(a^Hz):~z^HBz+2\Re(b^Hz)+c\le(=) 0,~h^Hz+d\le0,~z\in \mathbb{C}^n\}
\]
with $h\in \mathbb{C}^n$ and $d\in \mathbb{R}$
can be solved efficiently. More precisely, we first check whether there is a global minimizer  of  $(\rm GT^{c})$ (or$(\rm GT_e^{c})$) satisfying $h^Hz+d\le0$, and if not then it is sufficient to solve
\begin{eqnarray*}
{\min} \{z^HAz+2\Re(a^Hz):~z^HBz+2\Re(b^Hz)+c\le(=)  0,~h^Hz+d=0,~z\in \mathbb{C}^n\},
\end{eqnarray*}
which can be reformulated as $(\rm GT^{c})$ (or $(\rm GT_e^{c})$) in $\mathbb{C}^{n-1}$ via variable reduction.

\section{Conclusions}\label{sec5}
Trust-region subproblem (T) is fundamental in the area of nonconvex optimization. It is polynomially solvable based on the inherent hidden convexity. (T) is shown to have at most one local nonglobal minimizer, which has been fully characterized recently. We study in this paper the generalized trust-region subproblem (GT).  It  inherits the hidden convexity and the global optimality condition from (T). Different from (T), we show that there may be more than  one local nonglobal minimizer of  (GT). Our main contribution in this paper is to prove that, at any local nonglobal minimizer of (GT),  not only the strict complementarity condition holds, but also  the standard second-order sufficient optimality condition surprisingly remains necessary. As a corollary, finding all local nonglobal minimizers of (GT) or proving the nonexistence can be done in $O(n^4)$ time, where $n$ is the dimension. It is unknown whether there is an algorithm with lower complexity. We also study (GT) in complex domain, and show that it has no local nonglobal minimizer. It is the future work to further tighten the upper bound of the number of local nonglobal minimizers of real-valued (GT). Though we have proved that two is tight for two-dimensional problems, in general, the tight upper bound remains unknown.

\end{document}